\documentclass[a4paper]{amsart}

\usepackage{amsmath, amsfonts, amsthm , amssymb, amscd}
\usepackage{graphicx}
\usepackage{pst-grad} 

\usepackage{url}
\usepackage{pstricks}
\usepackage{xcolor}

 \newtheorem{thm}{Theorem}[section]
 \newtheorem{prop}[thm]{Proposition}
 \newtheorem{lem}[thm]{Lemma}
 \newtheorem{cor}[thm]{Corollary}

\theoremstyle{definition} 
 \newtheorem{dfn}{Definition}[section]
 \newtheorem{exm}[thm]{Example}
 \newtheorem{rem}[thm]{Remark}
\numberwithin{equation}{section}

\title[Chromatic numbers for facet colouring of generalised associahedra]{Chromatic numbers for facet colouring of some
generalised associahedra}
\author{Dj.\ Barali\' c, J.\ Ivanovi\' c and Z.\ Petri\' c}
\address{ \scriptsize{Mathematical Institute SANU\\ Knez Mihailova 36, p.f.\ 367\\
        11001 Belgrade, Serbia}}
\email{djbaralic@mi.sanu.ac.rs, zpetric@mi.sanu.ac.rs}
\address{\scriptsize{University of Belgrade, Faculty of Architecture\\ Bulevar Kralja Aleksandra 73/II\\ 11000 Belgrade,
Serbia}}
\email{jelena.ivanovic@arh.bg.ac.rs}

\date{}

\allowdisplaybreaks
\begin{document}

\begin{abstract} The chromatic number related to a colouring of facets of certain classes of generalised associahedra is studied. The exact values are obtained for permutohedra, associahedra and simple permutoassociahedra, while lower and upper bounds are established for cyclohedra and stellohedra. The asymptotic values of the chromatic numbers for associahedra, cyclohedra and simple permutoassociahedra are given.

\vspace{.3cm}

\noindent {\small {\it Mathematics Subject Classification} ({\it
        2020}): 52B05, 52B12, 05C15
}

\vspace{.5ex}

\noindent {\small {\it Keywords$\,$}:  permutohedron, associahedron, cyclohedron, stellohedron, simple permutoassociahedron }
\end{abstract}

\maketitle

\section{Introduction}

James Stasheff, in his remarkable paper \cite{S63} from 1963, brought to the scientific community's attention a simple polytope (in dimension $n$, every vertex is incident with exactly $n$ facets) with outstanding combinatorial properties, today known as \textit{the associahedron} or \textit{Stasheff polytope}. The fact that this polytope was constructed by Dov Tamari in 1951 (see \cite{T51}) was pointed out by Stasheff himself in \cite{S12}. In the 1990s, several other constructions of series of polytopes under the common name ``generalised associahedra'' appeared in the literature, finding numerous applications in algebraic geometry \cite{GKZ94}, the theory of operads \cite{S97}, cluster algebras \cite{FZ02} and knot theory \cite{BT94}.

An essential class of generalised associahedra is the class of nestohedra. It could be introduced either geometrically by using a Minkowski sum of simplices, or combinatorially by relying on concepts of nested sets and building sets. Feichtner and Kozlov, \cite{FK04}, defined the notions of
building sets and nested sets in a pure combinatorial manner. Feichtner and Sturmfels, \cite{FS05}, and Postnikov, \cite{P09}, used these concepts to describe the face lattices of a family of simple polytopes named nestohedra in \cite{PRW08}. An inductive approach that leads to the same family of polytopes is given in \cite{DP10c}, and Carr and Devadoss used graphs in \cite{CD06} to define a related family of polytopes. A possibility to iterate this construction and to extend this family is given in \cite{P14} (see also \cite{BIP17} and  \cite{I20}). All these families of simple polytopes lie at the crossroads of algebra, combinatorics, geometry, logic, topology, and other fields of mathematics and physics.

In the paper, we study the chromatic number of some nestohedra together with the simple permutoassociahedron concerning proper colourings of their facets, i.e. colourings such that two facets sharing a vertex have distinct colours.
It should be noted that
the colouring of the vertices of graphs is the most studied one in the literature; for example, the chromatic number of proper colourings of vertices of the associahedron was studied in \cite{MPHHUW09} and \cite{ARSW18}.
Our definition of proper coloring applied to a simple polytope $P$ is equivalent to the standard proper colouring of the vertices of its dual polytope, which is simplicial (see e.g.\ \cite{Zig} for more details). However, we prefer to work with graph associahedra, instead of their dual polytopes, as the latter are much less known objects in mathematics.

Our investigation of colourings of the facets of simple polytopes in the families listed above is motivated by the fact that they are all Delzant polytopes (see \cite{Z06}, \cite{P09} and \cite{P14}). Recall that a polytope is called Delzant if, for each of its vertices, the outward normal vectors of the facets containing it can be chosen to make up an integral basis for $\mathbb{Z}^n$. The normal fan of a Delzant polytope is a complete non-singular fan, thus defining a projective toric manifold and a real toric manifold by the fundamental theorem. This fact opened the possibility of calculating some of their combinatorial invariants, such as their $\mathbf{h}$-vectors and the Betti numbers of the corresponding toric manifolds (see \cite{CP15} and \cite{CPP17}). However, a projective toric manifold is also a quasitoric manifold, while a real toric manifold is a small cover, a topological generalisation introduced by Davis and Januszkiewicz in \cite{DJ91}. This remarkable paper  traced the foundation of toric topology as a new mathematical discipline at the junction of equivariant topology, algebraic geometry, combinatorics and commutative algebra. Toric topology has found its applications in the chemistry of fullerenes \cite{BE17} and topological data analysis \cite{LU23}. The main objects of study here are topological spaces and manifolds with torus actions defined in combinatorial terms.

Among the most investigated and best-understood objects in toric topology are quasitoric manifolds. From a topological point of view, a quasitoric manifold $M^{2n}$ is a smooth closed $2n$-dimensional manifold with a smooth locally standard action of torus $T^n$ such that its orbit space is a simple convex $n$-polytope $P^n$ regarded as a manifold with corners. However, its combinatorial description provides further insights as it introduces a quasitoric manifold as a characteristic pair $(P^n, \Lambda)$ where $P^n$ is a simple polytope and $\Lambda$ is the characteristic map that assigns a vector of $\mathbb{Z}^n$ to each facet of $P^n$ so that for each its vertex, the assigned vectors of the facets containing it span an integral basis for $\mathbb{Z}^n$. A detailed exposition on the construction of a quasitoric manifold from a characteristic pair and the equivalence of topological and combinatorial definition may be found in \cite[Construction~5.12]{BP02}. The main problem in toric topology is the classification problem of quasitoric manifolds, which asks whether a given polytope $P^n$ appears as the orbit space of a quasitoric manifold and, if it appears, to determine all such quasitoric manifolds. Up to the  action of the symmetry group of combinatorial polytope $P^n$ and the general linear group $GL(n)$, the question reduces to describing all characteristic maps on the facets of $P^n$. The chromatic number for proper colouring of the facets of a simple $n$-polytope is the only known obstruction to the existence of a characteristic map on $P^n$; it is always greater than $n$ and less or equal to $2^n-1$ \cite[Example~3.3.4]{BP15}. However, it is unknown whether the upper bound in this inequality is sharp or not. In the literature, to the best of our knowledge, the families of polytopes admitting a characteristic map have small or linear values of the chromatic number \cite{BM22}.

The number of vectors of $\mathbb{Z}^n$ in the image of a characteristic map on $P^n$ is an upper bound for the chromatic number of a proper facet colouring of $P^n$. However, the characteristic map associated with projective toric manifolds arising from Delzant realizations of the permutohedron, the associahedron, the cyclohedron, the stellohedron, and the permutoassociahedron assign distinct vectors to distinct facets so they produce the coloring with the maximum number of colours. This fact motivated us to study the natural question of determining  their chromatic numbers as the first step to classifying the quasitoric manifolds over the nestohedra. The deep connection between topology and combinatorics of quasitoric manifolds has exciting applications in binary coding \cite{CLY18}. The results obtained in the paper for the associahedron, the cyclohedron, the stellohedron and the permutoassociahedron are the first to point out classes of simple polytopes admitting a characteristic map where the chromatic number is not a linear function of the dimension.

\section{Preliminaries}

Throughout the text, the cardinality of a set $X$ is denoted by $|X|$. For $n\geq 1$, the sets $\{1,\ldots,n\}$ and $\{0,1,\ldots,n\}$ are denoted by $[n]$ and $[n]_0$, respectively. The subset relation is denoted by $\subseteq$, while the \emph{proper} subset is denoted by $\subset$. Also, by \emph{comparability} of sets, we mean the comparability with respect to $\subseteq$.

There are several definitions of building and nested sets, and we recall one stemming from~\cite{FS05}.

\begin{dfn}\label{d1}
A collection $\mathcal{B}$ of nonempty subsets of $[n]_0$ containing all singletons $\{x\}, x\in [n]_0$, and satisfying that
for any two sets $S_1, S_2\in \mathcal{B}$ such that $S_1\cap S_2
\neq\emptyset$, their union $S_1\cup S_2$ also belongs to
$\mathcal{B}$, is called a \textit{building set}.

For a family of sets $N$, we say that $\{X_1,\ldots,X_m\}\subseteq
N$ is an $N$-\emph{antichain}, when $m\geq 2$ and $X_1,\ldots,X_m$
are mutually incomparable. Let $\mathcal{B}$ be a building set such that $[n]_0\in \mathcal{B}$. We say that $N\subseteq \mathcal{B}$ is a \emph{nested set} when the union of every $N$-antichain is not an element of $\mathcal{B}$.

A subset of a nested set is again a nested set, hence the nested sets with respect to $\mathcal{B}$ make a simplicial complex. The link of $[n]_0$ in this complex (i.e.\ the set of all nested sets $N$ such that $[n]_0\not\in N$) is the \emph{nested-set complex} of $\mathcal{B}$.

We say that a polytope $P$ (\emph{geometrically}) \emph{realises} a simplicial complex $K$, when the semilattice obtained by removing the bottom (the empty set) from the face lattice of $P$ is
isomorphic to $(K,\supseteq)$. Note that $P$ must be simple, and for every $\mathcal{B}$, there is a polytope, called \emph{nestohedron}, that geometrically realises the corresponding nested-set complex (see \cite[Theorem~7.4]{P09}, or \cite[Theorem~9.10]{DP10c}).
\end{dfn}

Consider the following four graphs with $[n]_0$ as the set of
vertices: the complete graph on $[n]_0$, the path $0-\ldots-n$,
the cycle $0-\ldots-n-0$ and the star with $n$ edges connecting
the vertex $0$ with all the other vertices. For every such graph
$\Gamma$, the set of all nonempty and
connected subsets of vertices in $\Gamma$ make a building set. Each of these building sets gives rise to a nested-set complex,
which can be realised as an $n$-dimensional simple polytope. In
this way the $n$-dimensional \emph{permutohedron},
\emph{associahedron}, \emph{cyclohedron} and \emph{stellohedron}
(we prefer its entirely Greek name \emph{astrohedron}), correspond to
the complete graph on $[n]_0$, the path $0-\ldots-n$, the cycle
$0-\ldots-n-0$ and the star with $n$ edges, respectively.

\begin{exm}
The 3-dimensional associahedron corresponds to the path graph $0-1-2-3$, and its building set $\mathcal{B}$ is the following:
\[
\{\{0\},\{1\},\{2\},\{3\},\{0,1\},\{1,2\},\{2,3\},\{0,1,2\},\{1,2,3\}, \{0,1,2,3\}\}.
\]
Note that $\{\{1,2\},\{2,3\}\}$ is an $N$-antichain for $N=\{\{1\},\{1,2\},\{2,3\}\}$, but $N$ is not a nested set since $\{1,2\}\cup \{2,3\}=\{1,2,3\}\in \mathcal{B}$. On the other hand, $\{\{1\},\{0,1\},\{3\}\}$ is a nested set since neither $\{1,3\}$, nor $\{0,1,3\}$ belongs to $\mathcal{B}$. For $X=\{0,1,2\}$, the nested set $\{X\}$ corresponds to a facet of this polytope. This facet is a pentagon whose edges correspond to the nested sets
\[
\{\{0\},X\}, \{\{0,1\},X\},\{\{1\},X\}, \{\{1,2\},X\}, \{\{2\},X\},
\]
and whose vertices correspond to the nested sets
\[
\{\{0\},\{0,1\},X\}, \{\{0,1\},\{1\},X\}, \{\{1\},\{1,2\},X\},\{\{1,2\},\{2\},X\}, \{\{2\},\{0\},X\}.
\]
\end{exm}

Note that the facets of nestohedra always correspond to singleton nested sets of the form $\{X\}$, with $X$ a proper subset of $[n]_0$. Throughout the paper, a facet will be identified with $X\subset [n]_0$ such that the nested set $\{X\}$ corresponds to this facet. The pentagonal facet from the preceding example is identified with $X=\{0,1,2\}$.

\begin{dfn}
We say that two facets of a polytope are \emph{separated}
when they have no vertex in common.
\end{dfn}

\begin{rem}
Two facets of a realisation of a nested-set complex are separated if and only if there is no nested set that contains them both.
\end{rem}

\begin{prop}\label{tv1}
Two facets are separated if and only if they are incomparable and their union belongs to the building set.
\end{prop}

\begin{proof}
$(\Leftarrow)$ It is clear that if two facets are incomparable and their union belongs to the building set, then there is no nested set that contains them both.
\vspace{1ex}

$(\Rightarrow)$ By contraposition, if two facets $X$ and $Y$ are comparable or their union does not belong to the building set, then $\{X,Y\}$ is a nested set that contains them both.
\end{proof}

\begin{dfn}\label{bojenje}
For $\mathcal{F}$ being the set of facets of a given polytope, a function $f\colon\mathcal{F}\rightarrow[m]$ is a
\emph{proper colouring} of this polytope when every two facets of the same colour are separated. The \emph{chromatic number} of a polytope is the
minimal $m$ such that a proper colouring of its
facets with $m$ colours exists.
\end{dfn}

\section{Permutohedra}

The facets of the $n$-dimensional permutohedron are identified with the proper subsets of $[n]_0$. Two facets $X,Y\subset [n]_0$ have a common vertex, i.e.\ there is a nested set containing both if and only if $X$ and $Y$ are comparable, which entails that the colouring $f\colon\mathcal{F}\rightarrow [n]$ defined by $f(X)=|X|$ is proper. Since the chromatic number of an $n$-dimensional polytope has $n$ as the lower bound, this means that the chromatic number $\pi_n$ of the $n$-dimensional permutohedron is $n$.

\begin{figure}[h]
    \centering
    \begin{tabular}{cc}
\includegraphics[width=0.35\textwidth]{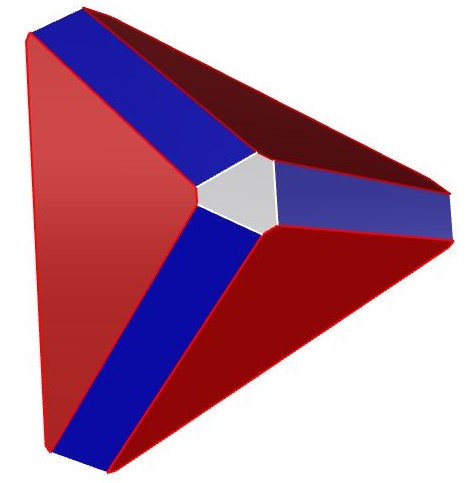}
&
\includegraphics[width=0.35\textwidth]{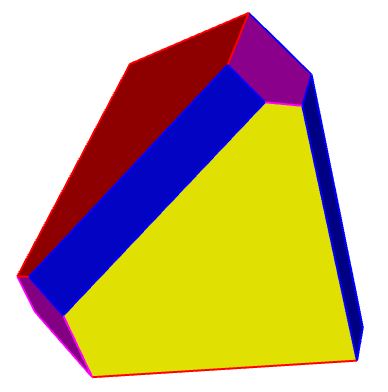}
\end{tabular}
    \caption{Proper colourings of the 3-dimensional permutohedron and the 3-dimensional  associahedron}
    \label{obojeniPA}
\end{figure}

\section{Associahedra}\label{asoc}

The building set for the $n$-dimensional associahedron consists of the sets of the form \[[a,b]=\{x\in [n]_0\mid a\leq x\leq b\},\] for $0\leq a\leq b\leq n$. The elements of the building set that are proper subsets of $[n]_0$ are called \emph{segments} of $[n]_0$. The facets of the $n$-dimensional associahedron are identified with the segments of $[n]_0$.

\begin{rem}\label{rem3}
The facets $[a,b]$ and $[c,d]$, with $a\leq c$, are incomparable and their union belongs to the building set \emph{if and only if} $a+1\leq c\leq b+1$ and $b<d$. In this case we say that $[a,b]$ \emph{precedes} $[c,d]$.
\end{rem}

As a corollary of Proposition~\ref{tv1} and Remark~\ref{rem3} we have the following.

\begin{cor}\label{cor1}
For mutually separated facets $X_1,\ldots,X_k$ and all
$i,j,l\in[k]$ we have:
\[X_i\cap X_j=X_i\cap X_l \Rightarrow j=l.\]
\end{cor}

\begin{prop}\label{tv2}
Let $f\colon \mathcal{F} \rightarrow [m]$ be a proper colouring of the $n$-dimensional associahedron.
Then for a facet $X$ of cardinality $\kappa$ the cardinality of
$\{Y\mid f(Y)=f(X)\}$ is at most $\min\{\kappa+1,n-\kappa+2\}$.
\end{prop}

\begin{proof}
Let $\mathcal{F}_X=\{Y\mid f(Y)=f(X)\}-\{X\}$. It suffices to show
that $|\mathcal{F}_X|\leq \min\{\kappa,n-\kappa+1\}$. Let
$g\colon\mathcal{F}_X\rightarrow X$ and
$h\colon\mathcal{F}_X\rightarrow [n]_0-X$ be defined as
\[
g(Y)=\left\{
\begin{array}{cl} \max(Y)+1 & \mbox{\rm if}\; Y\; \mbox{\rm precedes}\; X,\\[1ex]
\min(Y)-1 & \mbox{\rm if}\; X\; \mbox{\rm precedes}\; Y,
\end{array} \right .
\]
\[
h(Y)=\left\{
\begin{array}{cl} \min(Y) & \mbox{\rm if}\; Y\; \mbox{\rm precedes}\; X,\\[1ex]
\max(Y) & \mbox{\rm if}\; X\; \mbox{\rm precedes}\; Y.
\end{array} \right .
\]

That $h$ is one-one follows easily from Proposition~\ref{tv1}. By
Corollary~\ref{cor1}, we have that $g$ restricted to the members
of $\mathcal{F}_X$ that precede $X$ and $g$ restricted to the
members of $\mathcal{F}_X$ preceded by $X$ are one-one. It remains
to conclude, by Proposition~\ref{tv1}, that for every $Y$
preceding $X$ and every $Z$ preceded by $X$, we have $g(Y)>g(Z)$.
Hence, $g$ is one-one.
\end{proof}

Consider the following two sequences of integers $A_i$ and
$B_i$, for $2\leq i\leq \lceil n/2\rceil +1$.
\[
\begin{array}{l}
B_i=\left\{
\begin{array}{cl} 0 & \mbox{\rm if}\; i=2,\\[1ex]
n+3+B_{i-1}-(i-1)A_{i-1} & \mbox{\rm if}\; 2<i\leq \lceil n/2\rceil +1,
\end{array} \right .
\\[3ex]
A_i=\left\{
\begin{array}{cl} \left\lceil\frac{n+3+B_i}{i}\right\rceil & \mbox{\rm if}\;
n\; \mbox{\rm is even or}\; i\leq \lceil n/2\rceil,\\[1ex]
1 & \mbox{\rm if}\; n\; \mbox{\rm is odd and}\; i= \lceil n/2\rceil +1,
\end{array} \right .
\end{array}
\]

\begin{lem}\label{lem0}
There is a proper colouring of the $n$-dimensional associahedron with $\sum\limits_{i=2}^{\lceil n/2\rceil +1}A_i$ colours.
\end{lem}

\begin{proof}
If $n=2k-1$, then we consider the following sequences of facets
\[
\begin{array}{lc}
(2) & [0,0],\ldots,[n,n],[0,n-1],[1,n]
\\[1ex]
(3) &
[0,1],\ldots,[n-1,n],[0,n-2],[1,n-1],[2,n]
\\[1ex]
 & \vdots
\\[1ex]
(k) &
[0,k-2],\ldots,[k+1,n],[0,k],\ldots,[k-1,n]
\\[1ex]
(k+1) & [0,k-1],\ldots,[k,n],
\end{array}
\]
and if $n=2k$, then we replace the last two rows by
\[
\begin{array}{lc}
(k) &
[0,k-2],\ldots,[k+2,n],[0,k+1],\ldots,[k-1,n]
\\[1ex]
(k+1) &
[0,k-1],\ldots,[k+1,n],[0,k],\ldots,[k,n]
\end{array}
\]

The rows denoted by $(2)-(k)$ have $n+3$ members, while the cardinality of the last
row is $k+1$, i.e.\ $\lceil n/2\rceil
+1$, in the case $n=2k-1$, or $n+3$, in the case $n=2k$. Note that every $i$ consecutive members of the row denoted by $(i)$
are mutually separated and that the last $j$ members of the row
denoted by $(i)$, where $i>j$, together with the first $i-j$
members of the row denoted by $(i+1)$ are mutually separated.
Thus, we can properly colour the facets according to the following procedure: start with the first two facets in row (2) and use one colour for them, continue with the second pair of facets in this row and use another colour for this pair, and so on until $\lceil{(n+3)/2}\rceil=A_2$ colours are used and all the facets from the first row are coloured.

In the case $n=2k$, the last colour is used just for one (the last) facet in row~(2), and since this facet is separated with the first member of row~(3), one can colour this member with the same colour. The number of remaining facets in row~(3) is $n+3+B_3$, where $B_3$ counted as above is $n+3+B_2- 2 A_2=n+3-2 \lceil{(n+3)/2}\rceil$, which is $0$ when $n$ is odd, and $-1$ when $n$ is even. The remaining facets of row~(3) are coloured three consecutive in one colour, and again, one or two facets from row~(4) are borrowed and coloured with the last colour in case 3 does not divide $n+3+B_3$. This means that we have used $\lceil{(n+3+B_3)/3}\rceil=A_3$ new colours to finish with all the facets from row~(3).

We proceed with such a colouring showing that $\sum_{i=2}^{k}A_i$ colours is sufficient for rows (2)-(k). When we reach the last row, in the case $n=2k-1$, the members of this row are mutually separated and could be coloured with one colour, hence $A_{k+1}=1$. In the case $n=2k$, we proceed as with all the other rows, which means that $\lceil{(n+3+B_{k+1})/(k+1)}\rceil=A_{k+1}$ of new colours is sufficient for the remaining facets of the last row. Hence, we have a proper colouring with $\sum_{i=2}^{\lceil n/2\rceil +1}A_i$ colours.
\end{proof}

In order to show that there is no proper colouring with less colours than we used in Lemma~\ref{lem0}, let us consider a slightly more general and more intuitive problem. Instead of facets ordered in rows as in the proof of Lemma~\ref{lem0}, assume we have $l_2\geq 0$ balls of weight
$1/2$, $l_3\geq 0$ balls of weight $1/3$, and so on, up to $l_m\geq
0$ balls of weight $1/m$. (The balls here represent facets, the number $l_i$ represents the number of facets in row~(i), and the denominator in the weight of a ball represents $\min\{\kappa+1,n-\kappa+2\}$, where $\kappa$ is the cardinality of the corresponding facet.) These balls should be \emph{packed} into boxes (the boxes present colours) such that if $n$ balls $x_1,\ldots,x_n$ are packed in one box, then for $\mu(x)$ being the weight of $x$
\[
n\cdot\max\{\mu(x_1),\ldots,\mu(x_n)\}
\]
does not exceed 1.

\begin{rem}\label{rem0}
According to Proposition~\ref{tv2}, every proper colouring corresponds to a packing where $l_2=\ldots=l_k=n+3$ and $l_{k+1}=k+1$, when $n=2k-1$, or $l_{k+1}=n+3$ when $n=2k$. One will see below that the proof of Lemma~\ref{lem1} requires to allow all nonnegative values for $l_2,\ldots,l_m$.
\end{rem}

We are now in a position to generalise the proper colouring used in Lemma~\ref{lem0}, and to formulate this generalisation in terms of packing. Let $(\ast)$ be the following packing: start with $l_2$ balls of
weight $1/2$ and pack them into $\lceil{l_2}/{2}\rceil$ boxes,
where the last box, in case $l_2$ is odd, contains a ball of first
lower weight if there is any. The remaining balls of weigh $1/3$
are packed three per box, and one or two balls of first lower
weights are borrowed if 3 does not divide the number of these
remaining balls of weight $1/3$. We proceed analogously for all
the remaining balls of weight less than $1/3$.

\begin{rem}\label{rem1}
Our proper colouring from Lemma~\ref{lem0} corresponds to the packing $(\ast)$, where $l_2,\ldots,l_{k+1}$ are as in the first sentence of Remark~\ref{rem0}.
\end{rem}

\begin{dfn}
A packing is \emph{better} than another if it uses fewer boxes.
\end{dfn}

\begin{lem}\label{lem1}
There is no packing better than $(\ast)$.
\end{lem}

\begin{proof}
Suppose that $n=l_2+\ldots+l_m$ is the smallest number of balls
such that there exists a packing $(\ast\ast)$ better than
$(\ast)$. Let $i$ be the smallest integer such that $l_i>0$
($l_2=\ldots=l_{i-1}=0$).
Consider the first box $A$ in $(\ast)$. From the definition of this packing one concludes the following:
\begin{enumerate}
\item $A$ contains a ball of maximal weight (here $1/i$);
\item except for the balls of the lowest weight in $A$, it contains all the balls with higher weight;
\item the integer $i$ defined above is greater or equal to the number of balls in $A$;
\item in any correct packing of these $n$ balls, a box that contains a ball of the maximal weight cannot have more elements than $A$.
\end{enumerate}

Let us denote the balls of weight $1/i$ by $\bullet$, with the first lower weight by $\circ$, with the second lower weight by $\star$, and take the following example of $A$:
\[
\bullet\bullet\bullet\bullet\circ\circ\star\star\star.
\]
From above we conclude that $i\geq 9$ and that the number of balls denoted by $\bullet$ and $\circ$, among the $n$ balls, is 4 and 2, respectively and that a box containing a ball $\bullet$ in any other packing of these balls cannot have more than 9 elements.

Consider a box $B$ in $(\ast\ast)$ that contains $\bullet$. From the preceding sentence we know that $B$ contains not more than 9 balls and not more than 4 balls $\bullet$ and not more than 2 balls $\circ$. We have the following cases.

\vspace{1ex}

(\emph{i}) If $A$ and $B$ have the same content, then by eliminating $A$ from $(\ast)$ and $B$ from
$(\ast\ast)$, one obtains a packing better than $(\ast)$ with fewer than $n$ balls, which contradicts our assumption on minimality of $n$.

\vspace{1ex}

(\emph{ii}) If $B$ contains just the balls $\bullet$, then we move balls from other boxes in $(\ast\ast)$ to $B$ in order to make it identical to $A$. The obtained packing remains correct and this step does not increase the number of boxes, hence the new packing is still better than $(\ast)$. By proceeding as in (\emph{i}), we obtain a contradiction.

\vspace{1ex}

(\emph{iii}) If $B$ contains balls of not maximal weight, then we make some replacements in $(\ast\ast)$. In order to help the reader to follow this procedure, let us assume that $B$ is of the form:
\[
\bullet\bullet\circ\star\star\diamond,
\]
where the weight of the ball denoted by $\diamond$ is lower than the weight of a ball denoted by $\star$.

If we have fewer than 4 balls $\bullet$ in $B$, we replace a ball from $B$ that is of the first lower weight than $\bullet$ (in our case this is a ball $\circ$) with a ball $\bullet$ from some other box $C$ in $(\ast\ast)$. In our example $B$ becomes
\[
\bullet\bullet\bullet\star\star\diamond.
\]
The box $A$ is a witness that $B$ is not overburden, and a heavier ball is replaced by a lighter ball in $C$. Moreover, this step does not increase the number of boxes, hence the packing obtained by this replacement is still better than $(\ast)$. We continue with these replacements until we reach 4 balls $\bullet$ in $B$, or $B$ contains only the balls $\bullet$. In the latter case we proceed as in (\emph{ii}) in order to obtain a contradiction. In the former case, as in our example when $B$ becomes
\[
\bullet\bullet\bullet\bullet\star\diamond
\]
and it contains some balls other than $\bullet$, we continue with such replacements with the balls of the first lower weight in the place of $\bullet$, and in our example $B$ becomes
\[
\bullet\bullet\bullet\bullet\circ\circ.
\]
Eventually, we transform $B$ into an initial (not necessarily proper) segment of $A$. If so obtained $B$ is not the same as $A$, then as in (\emph{ii}), we move balls from other boxes in $(\ast\ast)$ to $B$ in order to make it identical to $A$. In our example, three balls $\star$ are moved from some other boxes to $B$, without replacement. We conclude, as above, that the obtained packing remains correct, and that the number of boxes in it does not increase. This brings us to the case (\emph{i}), which leads to a contradiction.
\end{proof}

From Lemmata~\ref{lem0}, \ref{lem1} and Remarks~\ref{rem0}, \ref{rem1} it is easy to conclude the following result concerning the chromatic number $\alpha_n$ of the $n$-dimensional associahedron.

\begin{thm}\label{teo1}
$\displaystyle\alpha_n=\sum_{i=2}^{\lceil n/2\rceil +1}A_i$.
\end{thm}

\vspace{2ex}

On the other hand, it is not easy to deduce an explicit formula for $\alpha_n$ from Theorem \ref{teo1} . However, it is clear that for all $2\leq i \leq \lceil n/2\rceil +1$ it holds that $-i+1\leq B_i \leq 0$. This observation, together with some elementary minimisation/maximisation principles, delivers the following inequalities.

\begin{equation}\label{ineq1}
\alpha_n \geq n \left(\frac{1}{1}+\frac{1}{2}+\cdots+\frac{1}{\lceil n/2\rceil}\right)+4\left(\frac{1}{1}+\frac{1}{2}+\cdots+\frac{1}{\lceil n/2\rceil}\right)- \frac{3n+8}{2}
\end{equation}
\begin{equation}\label{ineq2}
\alpha_n \leq n \left(\frac{1}{1}+\frac{1}{2}+\cdots+\frac{1}{\lceil n/2\rceil+1}\right)+3\left(\frac{1}{1}+\frac{1}{2}+\cdots+\frac{1}{\lceil n/2\rceil+1}\right)
\end{equation}

These inequalities indicate the importance of the harmonic number $\frac{1}{1}+\frac{1}{2}+\cdots+\frac{1}{n}$ for understanding the asymptotic behavior of the chromatic number $\alpha_n$. By Euler-Maclaurin formula, when $n$ tends to infinity, we have \begin{equation}\label{eq3}
\frac{1}{1}+\frac{1}{2}+\cdots+\frac{1}{n}=\ln n + \gamma + o(1),
\end{equation}
where $\gamma\approx 0.5772$ is the Euler-Mascheroni constant.

\begin{thm}\label{asi} When $n$ tends to infinity, we have that $\alpha_n \sim n \ln n$.
\end{thm}

\begin{proof} Using the inequalities~(\ref{ineq1}), (\ref{ineq2}), and the equality~(\ref{eq3}),  we obtain
\[
\frac{\ln \lceil n/2\rceil + \gamma}{\ln n}+ o(1)\leq \frac{\alpha_n}{n \ln n} \leq \frac{\ln (\lceil n/2\rceil +1) + \gamma}{\ln n}+ o(1).\]
Applying the sandwich theorem, it follows that
$$\lim_{n\to +\infty} \frac{\alpha_n}{n \ln n}=1,$$ as it was claimed.
\end{proof}

\section{Cyclohedra}\label{s_cikl}

In order to estimate the chromatic numbers of cyclohedra, we have to generalise some notions introduced at the beginning of Section~\ref{asoc}. Throughout this section, $\boxplus$ denotes $+_{n+1}$, i.e.\ the addition  modulo $n+1$.

\begin{dfn}\label{d_segment}
For $a,b\in [n]_0$, a \emph{segment} $[a,b]$ of $[n]_0$
is a proper subset of $[n]_0$ of the form
\[\begin{array}{ll}
  \{i\mid a\leqslant i\leqslant b\}& \text {if} \; a\leqslant b,
   \text {or} \\[1ex]
   \{i\mid a\leqslant i\leqslant n\}\cup \{i\mid 0\leqslant i\leqslant b\}              &  \text {if} \; a> b.
\end{array}\]
For two segments $[a,b]$ and $[c,d]$ of $[n]_0$ we say that:
\begin{list}{}{ \setlength{\leftmargin}{1.25em} \setlength{\parsep}{0.25em}}
\item{(\emph{i}) $[c,d]$ \emph{overlaps} $[a,b]$, when $c>a$ and these two segments are neither disjoint nor comparable;}
\item{(\emph{ii}) $[c,d]$ \emph{follows} $[a,b]$, when $b\boxplus 1=c$ (note that $\boxplus$ is $+_{n+1}$); }
 \item{(\emph{iii}) $[a,b]$ \emph{precedes} $[c,d]$, when $[c,d]$ overlaps $[a,b]$ or $[c,d]$ follows $[a,b]$.}
\end{list}
\end{dfn}

The facets of the $n$-dimensional cyclohedron correspond to the segments of $[n]_0$.
As in the case of associahedra, the following proposition can be proved analogously.

\begin{prop}\label{p_tv1ciklo}
Two facets are separated if and only if one precedes the other.
\end{prop}

For $n\geq 1$, let $\gamma_n$ be the chromatic number of the
$n$-dimensional cyclohedron.

\begin{prop}\label{t_ciklogornjagranica}
$\alpha_{n-1}+1\leqslant \gamma_n \leqslant \alpha_{n-1}+n$.
\end{prop}

\begin{proof}
{First of all, notice that $\mathcal{F}_a \subset \mathcal{F}$, where $\mathcal{F}_a$ and $\mathcal{F}$
are the facets of the $(n-1)$-dimensional associahedron and the $n$-dimensional cyclohedron, respectively. If $X\in \mathcal{F}-\mathcal{F}_a$ then $X=[0,n-1]$ or $n\in X$.
Since every element of $\mathcal{F}_a$ is a subset of $[0,n-1]$,
we have to use at least $\alpha_{n-1}+1$ colours in order to colour properly the facets of the cyclohedron. This extra colour can also be used for the facet $[n,n]$ preceded by $[0,n-1]$.

If among the remaining elements of $\mathcal{F}$ we observe two of the same cardinality, then, since they both contain $n$, one of them precedes the other, i.e., by Proposition~\ref{p_tv1ciklo}, they can be coloured with the same colour. Hence, if $h\colon\mathcal{F}_{a}\rightarrow[\alpha_{n-1}]$ is a proper colouring described in the proof of Lemma~\ref{lem0},
then the colouring $f\colon\mathcal{F}\rightarrow[\alpha_{n-1}+n]$ defined by
\[
    f(X) =
\begin{cases}
   h(X)& \text{ if } n\notin X  \text{ and } X\neq [0,n-1],\\[1ex]
  [\alpha_{n-1}]+1& \text{ if } X=[0,n-1],\\[1ex]
   [\alpha_{n-1}]+\lvert X\rvert, & \text { otherwise} .
\end{cases}
\]
is also proper.
}\end{proof}

Using the Sandwich theorem, from Theorem \ref{asi} and Proposition \ref{t_ciklogornjagranica} we obtain the asymptotic value of $\gamma_n$.

\begin{cor}\label{cor2} When $n$ tends to infinity, we have that $\gamma_n \sim n \ln n$.
\end{cor}

We are ready to introduce a sharper upper bound for $\gamma_n$. Consider the following two sequences  $A_i$ and $B_i$ , $2\leqslant i\leqslant \lceil \frac{n}{2}\rceil +1$, of integers.
\[
\begin{array}{l}
B_i=\left\{
\begin{array}{cl} 0, & \mbox{ if $ i=2$},\\[.5ex]
n+1+B_{i-1}-(i-1)A_{i-1}, & \mbox{ if $2<i\leqslant \lceil \frac{n}{2}\rceil +1$},
\end{array} \right .
\\[4ex]
A_i=\left\{
\begin{array}{cl} \left\lceil\frac{n+1+B_i}{i}\right\rceil, & \mbox{ if $ i\leqslant \lceil
\frac{n}{2}\rceil$},\\[.5ex]
1, & \mbox{ if $n$ is even and $i= \lceil \frac{n}{2}\rceil +1$},
\\[.5ex]0, & \mbox{ if $n$ is odd and $i= \lceil \frac{n}{2}\rceil +1$}.
\end{array} \right .
\end{array}
\]

\begin{thm}\label{t_cikologornja2}{
$\displaystyle \gamma_n\leqslant \sum_{i=2}^{\lceil  \frac{n}{2}\rceil +1}(A_i+1)$.
}
\end{thm}

\begin{proof}
 If $n=2k-1$, then we consider the following sequences of facets
\[
\begin{array}{lll}
(2) & [0,0],\ldots,[n,n] & [0,n-1],[1,n],\ldots,[n,n-2]
\\
(3) &
[0,1],\ldots,[n,0] & [0,n-2],[1,n-1],\ldots,[n,n-3]
\\[.5ex]
 & \vdots
\\
(k) &
[0,k-2],\ldots,[n,k-3] & [0,k],\ldots,[n,k-1]
\\
(k+1) & [0,k-1],\ldots,[n,k-2].
\end{array}
\]
The rows denoted by $(2)-(k)$ are separated in two blocks, the \emph{left-block} and the \emph{right-block}, while the last
row has only the left-block. Each block in each row has $n+1$ members.

Note that all the members of the last row
are mutually separated, and hence, they can be coloured with the same colour.
Also, note that every $i$ consecutive members of the left-block of the row denoted by $(i)$ are mutually separated, and that the last $j$, where $i>j$, members of the left-block in that row together with the first $i-j$ members of the left-block of the row denoted by $(i+1)$ are mutually separated. It means that we can colour the members of the left-blocks of the rows as in the proof of Lemma~\ref{lem0} and the packing procedure $(\ast)$, where $l_2=\ldots=l_{k+1}=n+1$. Thus, it is
straightforward to see that we need
\[1+\sum\limits_{i=2}^{k}A_i\] colours for a proper colouring of all the members of the left-blocks.
Since for every $2\leq i \leq k$, the members of the right-block of the row denoted by $(i)$ are mutually separated, they can be coloured with the same colour.
This means that we have coloured properly all the facets using exactly $$k+\sum\limits_{i=2}^{k+1}A_i$$ colours.

If $n=2k$, then we consider the same sequences of facets as
above save that the rows denoted by $(k)$ and $(k+1)$ are of the
form:
\[
\begin{array}{lll}
(k) &
[0,k-2],\ldots,[n,k-3] & [0,k+1],\ldots,[n,k]
\\
(k+1) &
[0,k-1],\ldots,[n,k-2] & [0,k],\ldots,[n,k-1].
\end{array}
\]
Observe that the members of the right-block of the row denoted by $(k+1)$ are mutually separated, and can be coloured with the same colour. By repeating the procedure as above, we use
$\sum\limits_{i=2}^{k+1}A_i$ colours for a proper colouring of all the members of the left-blocks and then $k$ colours for all the members of the right-blocks, which delivers the same conclusion as above.
\end{proof}

It can be verified that Theorem~\ref{t_cikologornja2} gives a significantly better upper bound than Proposition~\ref{t_ciklogornjagranica}.
For example, by Proposition~\ref{t_ciklogornjagranica} we have $5 \leq\gamma_4 \leq 8$, $7 \leq\gamma_5 \leq 11$, $9 \leq\gamma_6 \leq 14$, $ 11\leq\gamma_7 \leq 17$, $360 \leq\gamma_{100} \leq 459$, $5809 \leq\gamma_{1000} \leq 6808$, $37028 \leq\gamma_{5000} \leq 42027$, $80967 \leq\gamma_{10000} \leq 90966$,
while by Theorem~\ref{t_cikologornja2} we obtain $\gamma_4 \leq 6$, $\gamma_5 \leq 8$, $\gamma_6 \leq 10$, $\gamma_7 \leq 13$, $\gamma_{100} \leq 406$,$\gamma_{1000} \leq 6302$, $\gamma_{5000} \leq 39520$, $\gamma_{10000} \leq 85958$.

Still, for some $n$, one can improve the upper bounds for $\gamma_n$  given in Theorem~\ref{t_cikologornja2}.
Namely, for $2\leqslant i\leqslant \lceil  \frac{n}{2}\rceil$, let
$$[a_1,b_1]=X_1,\ldots,[a_i,b_i]=X_i$$ be $i$ consecutive facets coloured with the same colour according to the procedure from the proof of Theorem~\ref{t_cikologornja2}, i.e. $X_1,\ldots,X_i$ are
$i$ consecutive members of the left-block of the row denoted by $(i)$, or the last $j$ members of the left-block of that row,
together with the first $i-j$ members of the left-block of the row denoted by $(i+1)$. Recall that $\boxplus$ denotes $+_{n+1}$.
Consider the following sequence of $i-1$ facets
\[
Y_1=[a_i\boxplus 1,a_1\boxplus n],\; Y_2=[a_i\boxplus 2,a_1], \ldots,Y_{i-1}=[a_i\boxplus (i-1),a_1\boxplus (i-3)].
\]
Since $a_i=a_1\boxplus (i-1)$, the cardinality of these facets is $n-i+1$, which implies that they belong to the right-block of the row denoted by $(i+1)$
(when $n=2k-1$ and $i=k$, they belong to the left-block of the last row).
Let us show that the facets $X_1,\ldots,X_i,Y_1,\ldots, Y_{i-1}$ can be coloured by the same colour, i.e. that the colour used for the facets $X_1,\ldots,X_i$, can be also used for the facets $Y_1,\ldots,Y_{i-1}$.
Since $a_1\boxplus n\boxplus 1=a_1$, it is obvious that for every $1\leq l\leq i-1$ the facet $X_l$ follows the facet $Y_l$.
If $X_2$ belongs to row $(i)$, then the following equations hold:
\[
b_2\boxplus 1=b_1\boxplus 2=a_1\boxplus (i-2)\boxplus 2=a_1\boxplus (i-1)\boxplus 1=a_i\boxplus 1,
\]
which entails that $Y_1$ follows $X_{2}$. Otherwise, if $X_2$ belongs to row $(i+1)$,
then $b_2=a_2\boxplus (i-1)=a_1\boxplus(i-1)\boxplus 1=a_i\boxplus 1$, and $Y_1$ overlaps $X_2$. Hence, $X_2$ precedes $Y_1$ and similarly, for every  $1\leq l\leq i-1$,
the facets $X_{l+1}$ precedes $Y_l$. All this, together with the following equations
 \[ a_1\boxplus (i-3)=a_1\boxplus n\boxplus (i-2)=b_1\boxplus n,
\]
implies that for every $m \notin\{l,l+1\}$, either $Y_l$ precedes $X_m$ or $X_m$ precedes $Y_l$.
We conclude that the facets $X_1,\ldots, X_i, Y_1,\ldots, Y_{i-1}$ are mutually separated, and they can be coloured with the same colour.

Note that the segments $Y_0=[a_i,a_1\boxplus (n-1)]$ and  $X_i$ are mutually comparable, as well as the segments $Y_{i}=[a_i\boxplus i,a_1\boxplus (i-2)]$ and $X_1$,
and hence, by Proposition~\ref{p_tv1ciklo}, $Y_0$ and $Y_i$ cannot be included in this set of mutually separated facets.

It remains to conclude that during the colouring of $i$ consecutive members of the left-block of row $(i)$ or of the last $j$ members of the left-block of this row, together with the first $i-j$ members of the left-block of row $(i+1)$, we can colour with the same colour exactly $i-1$ corresponding members of the right-block of row $(i+1)$.
This potentially gives a procedure to spend less than $\lceil \frac{n}{2}\rceil$ colours for the remaining, still uncoloured facets.

Whether that possibility can really be exploited is independent of the number of the remaining facets.
It depends more on what precisely the remaining facets are, i.e., it depends on their characteristics determined by the nature of the dimension~$n$.
For example, assume that for some $1\leq m\leq n+1$, the $m$-th member of each of the right-blocks remains uncoloured. Since they are mutually comparable,
they are not separated and have to be coloured with different colours---there is no improvement.

\begin{exm}
Let $f$ be the proper colouring of the $7$-dimensional cyclohedron described in the proof of Theorem~\ref{t_cikologornja2}, i.e.
\begin{align*}
& f([0,0])=f([1,1])=1,
\\
& f([2,2])=f([3,3])=2,
\\
& f([4,4])=f([5,5])=3,
\\
& f([6,6])=f([7,7])=4,
\\
& f([0,1])=f([1,2])=f([2,3])=5,
\\
&  f([3,4])=f([4,5])=f([5,6])=6,
\\
&  f([6,7])=f([7,0])=f([0,2])=7,
\\
& f([1,3])=f([2,4])=f([3,5])=f([4,6])=8,
\\
&   f([5,7])=f([6,0])=f([7,1])=f([0,3])=9,
\\
& f([0,6])=f([1,7])=\ldots=f([7,5])=10,
\\
& f([0,5])=f([1,6])=\ldots=f([7,4])=11,
\\
& f([0,4])=f([1,5])=\ldots=f([7,3])=12,
\\
& f([1,4])=f([2,5])=\ldots=f([7,2])=13.
\end{align*}

By the procedure described above, there is a proper colouring $h$, which saves one colour, and is defined as follows:
\begin{align*}
& h([0,0])=h([1,1])=h([2,7])=1,
\\
& h([2,2])=h([3,3])=h([4,1])=2,
\\
& h([4,4])=h([5,5])=h([6,3])=3,
\\
& h([6,6])=h([7,7])=h([0,5])=4,
\\
& h([0,1])=h([1,2])=h([2,3])=h([3,7])=h([4,0])=5,
\\
& h([3,4])=h([4,5])=h([5,6])=h([6,2])=h([7,3])=6,
\\
& h([6,7])=h([7,0])=h([0,2])=h([1,5])=h([2,6])=7,
\\
& h([1,3])=h([2,4])=h([3,5])=h([4,6])=h([5,0])=h([6,1])=h([7,2])=8,
\\
& h([5,7])=f([6,0])=h([7,1])=h([0,3])=h([1,4])=h([2,5])=h([3,6])=9,
\\
& h([0,6])=h([1,7])=\ldots=h([7,5])=10,
\\
& h([1,6])=f([3,0])=f([5,2])=f([7,4])=11,
\\
& h([0,4])=f([5,1])=f([4,7])=12.
\end{align*}

One can notice that the total number of colours used for the colouring $h$ is equal to $\alpha_7$, i.e. to the minimal number of colours needed for proper colouring of the $7$-dimensional associahedron. On the other hand, it can be verified that when $n=8$, no colour can be saved,
while, when $n=9$, this procedure saves one colour.
However, we can say nothing more about the relation between $\alpha_n$ and $\gamma_n$, except that these two chromatic numbers are equal for $n\leq 5$.
\end{exm}

\begin{figure}[h]
    \centering
    \begin{tabular}{cc}
\includegraphics[width=0.4\textwidth]{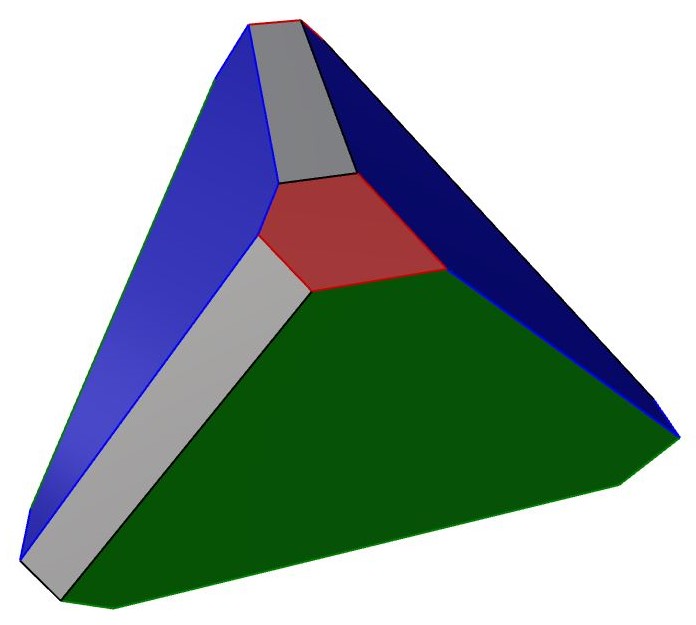}
&
\includegraphics[width=0.4\textwidth]{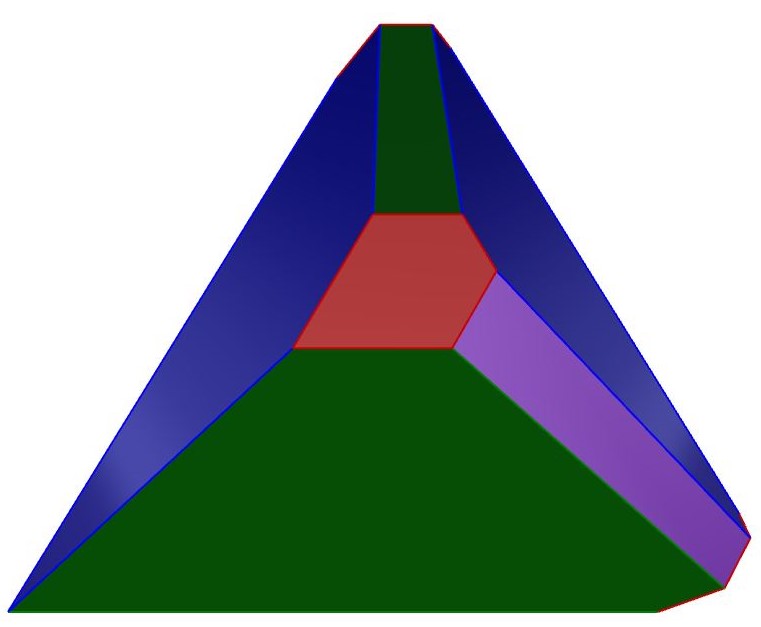}
\end{tabular}
    \caption{Proper colourings of the 3-dimensional cyclohedron and the 3-dimensional stellohedron}
    \label{obojeniCZ}
\end{figure}

\section{Stellohedra}

The facets of the $n$-dimensional stellohedron correspond to the proper subsets
of $[n]_0$ containing $0$ together with the singletons $\{1\},\ldots , \{n\}$. Two facets $X$ and
$Y$ have a common vertex if and only if they are singletons among $\{1\},\ldots, \{n\}$ or comparable.
Let $\sigma_n$  denote the chromatic number of the $n$-dimensional  stellohedron. In order
to give the upper bound for $\sigma_n$ we introduce the following function on
\[
\mathcal{H}=\{Y\subseteq \mathbf{N} \mid 0\in Y \text{ and } Y \text{ is finite}\}.
\]

By induction on $k =  \lvert Y- \{0\}  \rvert$,
we define $h: \mathcal{H} \rightarrow \mathbf{N}^+$ and the set $ F_Y \subseteq \mathbf{N}^+$, and
we show that
$$
\begin{array}{ll}
(\ast) &  h(Y)\notin F_Y \text{ and }   F_Y \cup \{h(Y)\} \text{ consists of the elements of  }
\\[.5ex] & (Y)_1^0=(Y- \{0\}) \cup \{1\}\text{ and at most k initial elements of } \mathbf{N}^+-(Y)_1^0.
\end{array}$$

For $k=0$, we define $F_{\{0\}}=\emptyset$, $ h(\{0\})=1$ and it is straightforward to see that
this satisfies $(\ast)$. Assume that $F_Y$ and $h(Y)$, which satisfy $(\ast)$, are defined for every
$Y \in \mathcal{H}$ such that $\lvert Y- \{0\}  \rvert<k$. For $Y=\{0,i_1,\ldots,i_k\}$, let
$Y_j=Y-\{i_j\}$ and let $Z_j=F_{Y_j}\cup \{h(Y_j)\}$, $1\leq j\leq k$. We define
\[
F_Y=(Y)_1^0 \cup (\bigcup \{Z_j \mid 1 \leq j \leq k\}), \quad \quad h(Y)=\min (\mathbf{N}^+-F_Y).
\]

For example, it follows that $F_{\{0,1\}}=\{1\}$, $h(\{0,1\})=2$,  $F_{\{0,2\}}=\{1,2\}$, $h(\{0,2\})=3$,
and for $i\geq 3$, $F_{\{0,i\}}=\{1,i\}$, $h(\{0,i\})=2$. Also, $F_{\{0,4,7\}}=\{1,2,4,7\}$, $h(\{0,4,7\})=3$.

By the induction hypothesis, every $Z_j$ consists of the elements of $(Y_j)_1^0$ together with at most $k-1$ initial elements of $\mathbf{N}^+-(Y_j)_1^0$,
and these $k-1$ elements belong to the union of $(Y)_1^0$ and the set of initial $k-1$ elements of $\mathbf{N}^+-(Y)_1^0$.
Hence, $F_Y$ is the subset of this union, and it is straightforward to see that it consists of the elements of $(Y)_1^0$ and at most $k-1$ initial elements of $\mathbf{N}^+-(Y)_1^0$.
When we add $h(Y)$ to this set, we see that $(\ast)$  holds. Note that from $(\ast)$ it follows that $h(Y)\leq 2k+1$ when $\lvert Y-\{0\}\rvert=k$.

\begin{rem}\label{zvezdrem}
If $Y,Z \in \mathcal{H}$ and $Z \subset Y$, then $h(Z)\in F_Y$, hence $h(Y)\neq h(Z)$. If $0\neq i \in Y \in \mathcal{H}$, then $i \in F_Y$, hence $h(Y)\neq i$.
\end{rem}

For $\mathcal{F}$ being the set of facets of the $n$-dimensional stellohedron, consider the following function
$$f:\mathcal{F}\rightarrow[2n-1- \lfloor {\frac{n-1}{2}}\rfloor].$$
Let $f(\{0\}=f(\{1\}))=1$ and for $2 \leq i \leq n$, let $f(\{i\})=i$. For $1 \leq k \leq \lfloor {\frac{n-1}{2}}\rfloor $ and the face $Y$ such that $\lvert Y-\{0\}\rvert = k$, let $f(Y)=h(Y)$
(note that in this case $h(Y)\in [n]$). For $k>\lfloor {\frac{n-1}{2}}\rfloor$ and $Y$ such that $\lvert Y-\{0\}\rvert=k$, let $f(Y)=n+k-\lfloor {\frac{n-1}{2}}\rfloor$.

\begin{thm}\label{steloup}
The function $f$ is a proper colouring.
\end{thm}
\begin{proof}
It suffices to see that for two faces $Y$ and $Z$ such that $Z \subset Y \in \mathcal{H}$ and $\lvert Y-\{0\}\rvert \leq \lfloor {\frac{n-1}{2}}\rfloor$ we have that $f(Y)\neq f(Z)$,
which follows from Remark~\ref{zvezdrem}.
\end{proof}

The stellohedron has pentagonal 2-faces, so by \cite[Theorem~16]{Jos02} it requires at least $n+1$ colors for a proper colouring of its facets. This observation and Theorem \ref{steloup} are summarized in the next corollary, with an open question whether the lower bound can be further improved for sufficiently large $n$.

\begin{cor} $n+1 \leq \sigma_n \leq 2n-1-\lfloor {\frac{n-1}{2}}\rfloor .$
\end{cor}

\section{Permutoassociahedra}

The last section of this paper is devoted to a family of polytopes whose construction is similar to the construction of nestohedra, but it is slightly more general. This family is presented in details in \cite{BIP17} and its geometric realisation through Minkowski sum of polytopes is obtained in \cite{I20}. The construction of such polytopes is a consequence of the possibility to iterate the standard nested-set construction (see \cite{P14}). One can notice an analogy between the polytopes obtained by the iterated nested-set construction and geometric realisations of finite simplicial complexes appearing in cluster algebras (see \cite{CFZ02} and \cite{FZ02}). We emphasize an evident correspondence between the notation used for the facets of both families. Our plan is to explore this analogy in a future work.

In order to introduce a family of polytopes obtained by an iterated nested-set construction, we have to define the notion of building sets of simplicial complexes and to adjust the definition of complexes of nested sets. These notions are just restrictions of the notions defined in \cite{FK04} for arbitrary finite-meet semilattices.

\begin{dfn}\label{d1}
A collection $\mathcal{B}$ of nonempty subsets of a finite set
$V$ containing all singletons $\{v\}, v\in V$ and satisfying that
for any two sets $S_1, S_2\in \mathcal{B}$ such that $S_1\cap S_2
\neq\emptyset$, their union $S_1\cup S_2$ also belongs to
$\mathcal{B}$, is called a \textit{building set} of
$\mathcal{P}(V)$. Let $K$ be a simplicial complex and let $V_1,\ldots,V_m$ be the
maximal simplices of $K$. A collection $\mathcal{B}$ of
simplices of $K$ is called a \emph{building set} of $K$ when for
every $1\leq i\leq m$, the collection
$$\mathcal{B}^{V_i}=\mathcal{B} \cap \mathcal{P}(V_i)$$ is a
building set of $\mathcal{P}(V_i)$.

We recall that for a family of sets $N$, we say that $\{X_1,\ldots,X_m\}\subseteq
N$ is an $N$-\emph{antichain}, when $m\geq 2$ and $X_1,\ldots,X_m$
are mutually incomparable.

Let $\mathcal{B}$ be a building set of a simplicial complex $K$.
We say that $N\subseteq \mathcal{B}$ is a \emph{nested set} with
respect to $\mathcal{B}$ when the union of every $N$-antichain is
an element of $K-\mathcal{B}$.
\end{dfn}

Let $C_0$ be the simplicial complex $\mathcal{P}([n]_0)-\{[n]_0\}$, i.e. the \emph{boundary complex}
$\partial \Delta^n$ of the abstract $n$-simplex $\Delta^n$, and let
$C_1$ be the simplicial complex of nested sets corresponding to
the complete graph on $[n]_0$, i.e.\ the one whose realisation is
the $n$-dimensional permutohedron. A nested set $A\in C_1$ is of the form $\{X_1,\ldots,X_l\}$, where
\[[n]_0\supset X_1\supset\ldots\supset X_l\supset\emptyset.\]

For a permutation
$\pi\colon [n]_0\rightarrow [n]_0$ and $A\in C_1$ let
\[
A^\pi=\{\{\pi(i_1),\ldots,\pi(i_k)\}\mid\{i_1,\ldots,i_k\}\in A\}.
\]
Let $\mathcal{B}_1\subseteq C_1$ be the building set of $C_1$ given by the family of all sets of the form
\[
\bigl\{
\{i_{k+l},\ldots,i_k,\ldots,i_1\},\ldots,\{i_{k+l},\ldots,i_k,i_{k-1}\},\{i_{k+l},\ldots,i_k\}
\bigr\},
\]
where $1\leq k\leq k+l\leq n$ and $i_1,\ldots,i_{k+l}$ are
mutually distinct elements of $[n]_0$. It is clear that
\[
(\dagger)\quad\quad\mbox{\rm if}\; A\in \mathcal{B}_1,\; \mbox{\rm
then}\; A^\pi\in \mathcal{B}_1\; \mbox{\rm for every permutation}\; \pi
\]
holds for $\mathcal{B}_1$.

The nested sets with respect to $\mathcal{B}_1$ are called
1-\emph{nested sets} and the complex of 1-nested sets is denoted
by $C$. The existence of a simple $n$-dimensional polytope that geometrically realises $C$ is shown in \cite[Theorem~5.2]{BIP17} (see also \cite{P14} and \cite{I20}). This polytope is called \emph{simple permutoassociahedron} and it is denoted by $PA_n$. The rest of this section examines its chromatic number.

\begin{rem} Simple permutoassociahedron $PA_n$ appears as the orbit space of torus action $T^n$ on a smooth projective toric variety which is a manifold. The object more known in the literature, the permutoassociahedron introduced by Kapranov in \cite{Kap93} is not simple, therefore does not posses these properties. However, finding the chromatic number of Kapranov's permutoassociahedron is still natural and interesting question for itself.
\end{rem}

\begin{dfn}\label{d3}
Let $\mathcal{E}$ be a subset of $\mathcal{B}_1$. We say that a function $f\colon \mathcal{E}\rightarrow[m]$ is \emph{proper},
when for all $A,B\in \mathcal{E}$ it holds that
\[
\mbox{\rm if}\; f(A)=f(B),\; \mbox{\rm then}\; A \; \mbox{\rm
and}\; B \; \mbox{\rm do not belong to the same 1-nested set}.
\]
\end{dfn}

Since the facets of the simple permutoassociahedron correspond to the elements of $\mathcal{B}_1$, it is not difficult to conclude that if $f\colon\mathcal{B}_1\rightarrow [m]$ is proper, then it corresponds
to a proper colouring of the facets of this polytope. Let $M$ be
the following maximal nested set in $C_1$
\[
\{\{n-1,\ldots,0\},\ldots,\{n-1,n-2\},\{n-1\}\}.
\]
Recall that $\mathcal{B}_1^M$ is $\mathcal{B}_1 \cap \mathcal{P}(M)$.

\begin{prop}\label{lokalizacija}
Let $f\colon\mathcal{B}_1\rightarrow[m]$ be a function such that $f(A)=f(A^\pi)$ for every $A\in \mathcal{B}_1$ and
every permutation $\pi$. If $f$ restricted to $\mathcal{B}_1^M$ is proper, then $f$ is proper.
\end{prop}

\begin{proof}
Let $A,B\in \mathcal{B}_1$ be such that both belong to a 1-nested
set $N$. The definition of $C_1$ guarantees that there is a
permutation $\pi$ such that $N\in
\mathcal{P}(M^{\pi^{-1}})$. Both $A^\pi$ and $B^\pi$ belong to
$N^\pi$, which is a 1-nested set by $(\dagger)$. Hence,
$f(A^\pi)\neq f(B^\pi)$ holds, and therefore $f(A)\neq f(B)$.
\end{proof}

\begin{prop}
If $A_1$ and $A_2$ are two different members of $\mathcal{B}_1^M$,
then for arbitrary permutations $\pi_1$ and $\pi_2$ we have $A_1^{\pi_1}\neq A_2^{\pi_2}$.
\end{prop}

\begin{proof}
Since $A_1\neq A_2$ and $A_1, A_2\in \mathcal{B}_1^M$, there is
$k\in\{1,\ldots,n\}$ such that one of $A_1, A_2$ contains a set from $M$ of cardinality $k$, while the other does not
contain a set of this cardinality. Hence, the same holds for $A_1^{\pi_1}$ i $A_2^{\pi_2}$.
\end{proof}

\begin{cor}\label{ekstenzija}
For an arbitrary function $f^M\colon \mathcal{B}_1^M\rightarrow
[m]$ there exists its extension $f\colon \mathcal{B}_1\rightarrow
[m]$, which satisfies $f(A^\pi)=f^M(A)$ for every $A\in
\mathcal{B}_1^M$ and every permutation $\pi$.
\end{cor}

Let $\delta_n$ be the chromatic number of the $n$-dimensional permutoassociahedron and recall that $\alpha_{n-1}$ is the chromatic number of the $(n-1)$-dimensional associahedron. Then we have the following.

\begin{thm}
$\delta_n=\alpha_{n-1}+1$
\end{thm}

\begin{proof}
Let $f$ be a proper colouring of $PA_n$ and let $A$ be a facet of $PA_n$ that is an $(n-1)$-dimensional associahedron (such a facet exists by the definition of $\mathcal{B}_1$). It is easy to see that $f$ induces a proper colouring of the facets of $A$, from which one concludes that $$\delta_n\geq\alpha_{n-1}+1.$$ To
prove the equality, by Proposition~\ref{lokalizacija} and
Corollary~\ref{ekstenzija}, it suffices to show the existence of a proper
$f^M\colon\mathcal{B}_1^M\rightarrow[\alpha_{n-1}+1]$.

Let $f^M$ be defined so that $f^M(M)=\alpha_{n-1}+1$, and the remaining facets belonging to $\mathcal{B}_1^M$ are coloured
according to the procedure introduced in Section~\ref{asoc}. For the sake of simplicity, one can rename the members of these facets to their minimal elements (e.g.\ the set $\{n-1,\ldots,n-k\}$ is renamed to
$n-k$) and precede exactly as in Section~\ref{asoc} in order to define a proper function from $\mathcal{B}_1^M-\{M\}$
to $[\alpha_{n-1}]$.
\end{proof}

\begin{figure}[h]
    \centering
\includegraphics[width=0.4\textwidth]{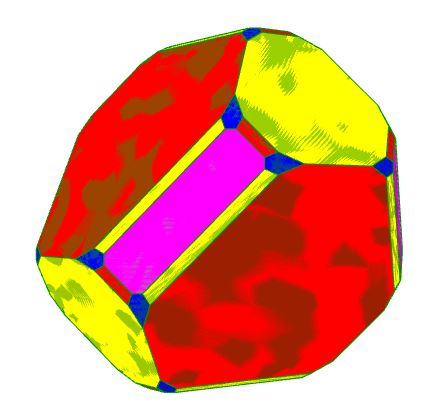}
    \caption{A proper colouring of the 3-dimensional permutoassociahedron}
    \label{obojeniperasoc}
\end{figure}

An immediate corollary of the above result and Theorem \ref{asi} is the following.

\begin{cor}\label{cor2} When $n$ tends to infinity, we have that $\delta_n \sim n \ln n$.
\end{cor}

\begin{center}\textmd{\textbf{Acknowledgements} }
\end{center}

\medskip
The authors are grateful to the anonymous referees for their helpful comments and valuable insights that improved considerably the quality of the manuscript. The authors were partially supported by the Ministry of Science, Technological Development and Innovation of the Republic of Serbia through the institutional funding of the Mathematical Institute SANU and the Faculty of Architecture, University of Belgrade.
The third author was supported by the Science Fund of the Republic of Serbia, Grant No. 7749891, Graphical Languages - GWORDS.

\begin{center}\textmd{\textbf{Data availability} }
\end{center}

\medskip
Data sharing not applicable to this article as no datasets were generated or analyzed during the current study.

\begin{center}\textmd{\textbf{Code availability} }
\end{center}

\medskip
Not applicable.

\begin{center}\textmd{\textbf{Conflict of interest} }
\end{center}

\medskip
The authors have no relevant financial or non-financial interests to disclose.

\end{document}